\documentclass[12pt]{amsart}
\usepackage{amsmath,amssymb,amsfonts,amsthm,amscd,indentfirst}
\usepackage{amsmath,amssymb,amsfonts,amsthm,amscd}
\usepackage{indentfirst}
\usepackage{hyperref}
\usepackage{graphicx}

\numberwithin{equation}{section}
\usepackage{xcolor}

\makeatletter
\@namedef{subjclassname@2020}{  \textup{2020} Mathematics Subject Classification}
\makeatother

\newtheorem{prop}{Proposition}[section]
\newtheorem{theo}[prop]{Theorem}
\newtheorem{lemm}[prop]{Lemma}

\newtheorem{rema}[prop]{Remark}

\def\<{\langle}
\def\>{\rangle}

\newcommand{\rr}{\mathbb{R}}

\begin{document}
\title{On Liouville's theorem for the Hessian quotient equation $\sigma_2/\sigma_1$}
\author{Siyuan Lu and Marcin Sroka}
\address{Department of Mathematics and Statistics, McMaster University, 1280 Main Street
West, Hamilton, ON, L8S 4K1, Canada}
\email{siyuan.lu@mcmaster.ca}
\address{Faculty of mathematics and computer science, Jagiellonian University, \L ojasiewicza 6, 30-348, Krak\'ow, Poland}
\email{marcin.sroka@uj.edu.pl}
\thanks{Research of the first author was supported in part by NSERC Discovery Grant. Research of the second author was supported in part by National Science Center of Poland grant no. 2025/57/B/ST1/00367.}

\begin{abstract} 
We prove Liouville's theorem for semi-convex entire solutions to Hessian quotient equation $\sigma_2/\sigma_1=1$ in $\rr^n$. The proof is based on the observation that after rewriting the quotient operator as the $\sigma_2$ operator, acting on a new function, one can refer to the recent result of Shankar and Yuan \cite{SY22} on Liouville's theorem for $\sigma_2$ equation. 
\end{abstract}

\maketitle

\section{Introduction}

Liouville's type theorem for Hessian equations
\begin{eqnarray}\label{generalhessian}
F(D^2u)=1,
\end{eqnarray} 
is a classic theme of research. In this paper we obtain such a theorem, alongside an interior estimate, for equation \eqref{generalhessian} when $F=\sigma_2/\sigma_1$. 

Recently there has been a substantial advance regarding regularity for (positive) Hessian quotient equations $\sigma_n/\sigma_k$, cf. \cite{L25,LT25}. It was sparked by the special concavity property for those operators discovered by the second named author and Guan \cite{GS25}. Motivated by this advances, we tried to obtain a similar concavity result for general operators $\sigma_k/\sigma_l$ when $1 \leq l < k < n$. As a side result, we observe that for the very particular case of $\sigma_2/\sigma_1$, one can substitute this sort of inequality by rewriting the resulting equation as $\sigma_2$ equation. Even though this rewriting, Lemma \ref{quotienttohessian} below, is surely known to experts, we are not aware of any of the presented consequences to be known. Especially, they all depend on the results for $\sigma_2$ equation proved only recently by Shankar and Yuan \cite{SY20,SY22,SY25}. 

In the context of quotient operators, till now, Liouville's theorem is known only for two cases: $\sigma_n/\sigma_k$ for $k=n-1,n-2$; and $\sigma_3/\sigma_1$ in dimensions $3$ and $4$. It is because $\sigma_n/\sigma_{n-1}$ is dual to $\Delta u$ and Liouville's theorem follows from classic result; $\sigma_n/\sigma_{n-2}$ is dual to $\sigma_2$ and Liouville's theorem follows from the result of Chang and Yuan \cite{CY}; $\sigma_3/\sigma_1$ in dimensions $3$ and $4$ can be written as special Lagrangian equation and Liouville's theorem follows from the result of Yuan \cite{Y}.

The other known case as far as we are aware, is the positive quotient operators $\sigma_n/\sigma_k$ for $1 \leq k <n$, however only under quadratic growth condition, cf. \cite{BCBJ03}. 

Our result below, Theorem \ref{Liouville21}, is for the other extreme and seems to be the first one, apart from the special Lagrangian one, in the case of operators $\sigma_k/\sigma_l$ when $k<n$. Importantly we do not put any extra assumption on the solution except, as is seen by Warren's example \cite{W16}, necessary one. As a simple consequence, Theorem \ref{Liouvillen-2n-1}, we obtain such a theorem also for strictly convex solutions to \eqref{generalhessian} when $F=\sigma_{n-1}/\sigma_{n-2}$.

\medskip

Interior estimate for \eqref{generalhessian} is another classic theme of research. By now the situation is clear for positive quotient operators by \cite{L25}. For $\sigma_k/\sigma_l$ when $k<n$, the only known case is $\sigma_3/\sigma_1$ in dimensions $3$ and $4$ by Chen, Warren and Yuan \cite{CWY} and Wang and Yuan \cite{WdY14}, again using the 
special Lagrangian structure of the equation. See also partial result \cite{LT26}. Our Theorem \ref{interiorestthm} for $\sigma_2/\sigma_1$ seems to be the first one, apart from the special Lagrangian one, in this direction. It was proved directly in \cite{L25} when $n=2$ and was noticed in \cite{S25}, again only when $n=2$, to follow from famous interior estimate of Heinz \cite{H} for the Monge-Amp\`ere equation in dimension two.        

\medskip

Through the paper we use the standard notation related to the Hessian equations which we briefly summarize. Let $\sigma_k$ for $1 \leq k \leq n$ denote the $k$-Hessian operator
\begin{eqnarray*}
\sigma_k(\lambda)=\sum_{1\leq i_1 < \cdots< i_k \leq n} \lambda_{i_1} \cdots\lambda_{i_k},
\end{eqnarray*}
for $\lambda \in \rr^n$. 

We denote by $\Gamma_k$ G\r{a}rding's $k$-cone
\begin{eqnarray*}
\Gamma_k=\{\lambda \in \rr^n \: |  \sigma_l(\lambda)>0,\forall 1\leq l \leq k\}.
\end{eqnarray*}

We say a function $u$ is an admissible solution for the $\sigma_k$ for $1 \leq k \leq n$ or $\sigma_k/\sigma_l$ for $1\leq l <k\leq n $ if $\lambda(D^2u)$ - the vector of eigenvalues of Hessian of $u$ - belongs to $\Gamma_k$. We understand that the notions of the operators $\sigma_k$, $\sigma_k/\sigma_l$ and the cones $\Gamma_k$ are extended to symmetric matrices through their eigenvalues as is standard to do. We say a function $u$ is semi-convex provided there exists $K>0$ such that $D^2u>-KI$, where $I$ is an identity matrix; convex when $D^2u \geq 0$; and strictly convex if, in a domain of consideration, $D^2u>0$.

\medskip

The organization of the paper is as follows. As we acknowledged above, the argument we present is based on rewriting the quotient operator as the Hessian operator which we present in Lemma \ref{quotienttohessian} below. Then, we present announced Liouville's theorems in Section 2. Interior estimate for the operator $\sigma_2/\sigma_1$ is derived in Section 3.

\begin{lemm}\label{quotienttohessian}
For $\lambda \in \Gamma_2$, the following equality holds
\begin{eqnarray*}
\sigma_2(\mu)=\frac{n}{2(n-1)}\cdot \left(\frac{\sigma_2}{\sigma_1}(\lambda)\right)^2
\end{eqnarray*}
for 
\begin{eqnarray*}
\mu=\lambda-\frac{\frac{\sigma_2}{\sigma_1}(\lambda)}{n-1}\cdot(1,\cdots,1)
\end{eqnarray*}
while $\mu \in \Gamma_2$. 
\end{lemm}

\begin{proof}
First, note that
\begin{align}\label{sigma1mu}
\sigma_1(\mu)=&\ \sum_{i=1}^n \left( \lambda_i-\frac{\frac{\sigma_2}{\sigma_1}(\lambda)}{n-1}\right)\\ \nonumber
=&\ \sum_{i=1}^n \lambda_i - \frac{n}{n-1}\cdot \frac{\sigma_2}{\sigma_1}(\lambda)\\ \nonumber
=&\ \sigma_1(\lambda)- \frac{n}{n-1}\cdot \frac{\sigma_2}{\sigma_1}(\lambda).
\end{align}

For later use, we observe further that
\begin{eqnarray}\label{sigma1mupositive}
\sigma_1(\mu) > 0.
\end{eqnarray}
This is because from Newton-Maclaurin inequality
\begin{eqnarray}\label{nm21}
\sigma_2(\lambda)\leq \frac{n-1}{2n} \sigma_1^2(\lambda)
\end{eqnarray}
and applying \eqref{nm21} in \eqref{sigma1mu} results in
\begin{eqnarray*}
\sigma_1(\mu) \geq \frac{1}{2}\sigma_1(\lambda)>0
\end{eqnarray*}
as required.

As for $\sigma_2(\mu)$, we compute
\begin{align}\label{sigma2muexpress}
\sigma_2(\mu)=&\ \sum_{1\leq i <j \leq n} \left( \lambda_i-\frac{\frac{\sigma_2}{\sigma_1}(\lambda)}{n-1}\right)\left( \lambda_j-\frac{\frac{\sigma_2}{\sigma_1}(\lambda)}{n-1}\right)
\\ \nonumber
=&\ \sum_{1\leq i <j \leq n} \lambda_i \lambda_j - \frac{\frac{\sigma_2}{\sigma_1}(\lambda)}{n-1} \cdot \sum_{1\leq i <j \leq n} (\lambda_i + \lambda_j) + \sum_{1\leq i <j \leq n}\left(\frac{\frac{\sigma_2}{\sigma_1}(\lambda)}{n-1}\right)^2 
\\ \nonumber
=&\ \sigma_2(\lambda) - \frac{\frac{\sigma_2}{\sigma_1}(\lambda)}{n-1} \cdot (n-1)\sigma_1(\lambda)+\frac{n(n-1)}{2}\cdot \left(\frac{\frac{\sigma_2}{\sigma_1}(\lambda)}{n-1}\right)^2\\\nonumber
=&\ \frac{n}{2(n-1)}\cdot \left(\frac{\sigma_2}{\sigma_1}(\lambda)\right)^2
\end{align}
as required. Coupling \eqref{sigma1mupositive} and \eqref{sigma2muexpress} shows that $\mu \in \Gamma_2$.
\end{proof}

\section{Liouville's theorem}

\begin{theo}\label{Liouville21}
Let $u\in C^\infty(\rr^n)$ be an admissible and semi-convex  solution of the equation
\begin{eqnarray}\label{entirequotient}
\frac{\sigma_2}{\sigma_1}(D^2u)=1,\quad \textit{in }\rr^n.
\end{eqnarray}
Then, $u$ is a quadratic polynomial.
\end{theo}

\begin{proof}
We introduce the function
\begin{eqnarray}\label{vdef}
v=u-\frac{1}{2(n-1)}|x|^2.
\end{eqnarray}

Then its Hessian $D^2v$ satisfies
\begin{eqnarray}\label{hessianrelation}
D^2v=D^2u-\frac{1}{n-1}I.
\end{eqnarray}

It follows from Lemma \ref{quotienttohessian}, applied to
\begin{eqnarray*}
\mu=\lambda(D^2v)
\end{eqnarray*} 
that
\begin{eqnarray*}
\sigma_2(\mu)=\frac{n}{2(n-1)}.
\end{eqnarray*}

Thus $v \in C^\infty(\rr^n)$ is, by \eqref{hessianrelation}, a semi-convex and admissible, by Lemma \ref{quotienttohessian}, solution of
\begin{eqnarray}\label{finalrewrition}
\sigma_2(D^2v)=\frac{n}{2(n-1)}.
\end{eqnarray}

By result of Shankar and Yuan \cite{SY22}, $v$ is a quadratic polynomial. Consequently, $u$ is a quadratic polynomial.
\end{proof}

\begin{rema} Warren’s rare saddle entire solution \cite{W16} verifies that also for \eqref{entirequotient} the semi-convexity assumption is necessary.
\end{rema}

As a corollary, we have the following result for $\sigma_{n-1}/\sigma_{n-2}$ equation.

\begin{theo}\label{Liouvillen-2n-1}
Let $u
\in C^\infty(\rr^n)$ be a strictly convex solution of the equation
\begin{eqnarray}\label{dualquotient21}
\frac{\sigma_{n-1}}{\sigma_{n-2}}(D^2u)=1,\quad \textit{in }\rr^n.
\end{eqnarray}
Then $u$ must be a quadratic polynomial.
\end{theo}

\begin{proof}
Let $w$ be the Legendre transform of $u$, cf. \cite{SY20} for details. Then 
\begin{eqnarray*}
D^2w(y)=(D^2u(x))^{-1}
\end{eqnarray*}
in the coordinates $y(x)=Du(x)$. Moreover $w$ is a convex solution of 
\begin{eqnarray*}
\frac{\sigma_2}{\sigma_1}(D^2w)=\frac{\sigma_2}{\sigma_1}\left((D^2u)^{-1}\right)=\frac{\sigma_{n-2}}{\sigma_{n-1}}(D^2u)=1,\quad \textit{in }\rr^n,
\end{eqnarray*}
by \eqref{dualquotient21}.

By Theorem \ref{Liouville21} above, $w$ is a quadratic polynomial. It follows that $u$ is a quadratic polynomial.

\end{proof}

\section{Interior estimate}

\begin{theo}\label{interiorestthm}
Let $u\in C^\infty(B_1)$ be an admissible solution of the equation
\begin{eqnarray*}
\frac{\sigma_2}{\sigma_1}(D^2u)=1,\quad \textit{in }B_1\subset\rr^n.
\end{eqnarray*}

Assume either $n=3,4$ or $n\geq 5$ with
\begin{align*}
\lambda_{min}(D^2u)-\frac{1}{n-1}\geq -c(n)\left(\Delta u-\frac{n}{n-1}\right) \: \text{for} \:\: c(n)=\frac{\sqrt{3n^2+1}-n+1}{2n}.
\end{align*}

Then, the following interior estimate holds 
\begin{eqnarray}\label{interiorest}
|D^2u(0)|\leq C,
\end{eqnarray}
where $C$ is a positive constant depending on $n$ and $\|u\|_{C^{0,1}(B_1)}$.
\end{theo}
\begin{proof}
We argue as in the proof of Theorem \ref{Liouville21} by defining $v$ by \eqref{vdef}. Then $v$ satisfies equation \eqref{finalrewrition}. We can then use Warren and Yuan's result \cite{WY} for $n=3$, and Shankar and Yuan's result \cite{SY25} for $n\geq 4$ to conclude the desired interior $C^2$ estimate \eqref{interiorest}.
\end{proof}


\begin{thebibliography}{99}
\bibitem{BCBJ03} 
J. Bao, J. Chen, B. Guan and M. Ji, 
\textit{Liouville property and regularity of a Hessian quotient equation}, 
Amer. J. Math. 125 (2003), no. 2, 301–316.

\bibitem{CY}
S.-Y. A. Chang and Y. Yuan,
\textit{A Liouville problem for the $sigma$-$2$ equation}, 
Discrete Contin. Dyn. Syst. 28 (2010), no. 2, 659–664.

\bibitem{CWY}
J. Chen, M. Warren and Y. Yuan, 
{\em A priori estimate for convex solutions to special Lagrangian equations and its application}, 
Comm. Pure Appl. Math. 62 (2009), no. 4, 583–595.

\bibitem{GS25} 
P. Guan and M. Sroka, 
\textit{A special concavity property for positive Hessian quotient operators}, 
Discrete Contin. Dyn. Syst., doi: 10.3934/dcds.2025181

\bibitem{H}
E. Heinz, 
{\em  On elliptic Monge-Amp\`ere equations and Weyl's embedding problem}, 
J. Analyse Math. 7 (1959), 1–52.


\bibitem{L25} 
S. Lu,  
\textit{Interior $C^2$ estimate for Hessian quotient equation in general dimension}, 
Ann. PDE 11 (2025), no. 2, Paper No. 17, 26 pp.

\bibitem{LT25} 
S. Lu and Y.-L. Tsai, 
\textit{Pogorelov type interior $C^2$ estimate for Hessian quotient equation and its application}, 
J. Reine Angew. Math. 831 (2026), 155-184.

\bibitem{LT26} 
S. Lu and Y.-L. Tsai, 
\textit{A note on interior $C^2$ estimate for general Hessian quotient equation}, 
preprint.

\bibitem{SY20} 
R. Shankar and Y. Yuan, 
\textit{Hessian estimate for semiconvex solutions to the $\sigma_2$ equation}, 
Calc. Var. Partial Differential Equations 59 (2020), no. 1, Paper No. 30, 12 pp.

\bibitem{SY22} 
R. Shankar and Y. Yuan, 
\textit{Rigidity for general semiconvex entire solutions to the $\sigma_2$ equation}, 
Duke Math. J. 171 (2022), no. 15, 3201–3214.

\bibitem{SY25} 
R. Shankar and Y. Yuan, 
\textit{Hessian estimates for the sigma-2 equation in dimension four}, 
Ann. of Math. (2) 201 (2025), no. 2, 489–513.

\bibitem{S25} 
M. Sroka, 
\textit{Remarks on Hessian quotient equations on Riemannian manifolds}, 
J. Funct. Anal. 289 (2025), no. 10, Paper No. 111123, 21 pp.

\bibitem{WdY14}
D. Wang and Y. Yuan, 
{\em Hessian estimates for special Lagrangian equations with critical and supercritical phases in general dimensions}, 
Amer. J. Math. 136 (2014), no. 2, 481–499. 

\bibitem{W16} 
M. Warren, 
\textit{Nonpolynomial entire solutions to $\sigma_k$ equations}, 
Comm. Partial Differential Equations 41 (2016), no. 5, 848–853.

\bibitem{WY}
M. Warren and Y. Yuan, 
\textit{Hessian estimates for the $sigma$-$2$ equation in dimension $3$}, 
Comm. Pure Appl. Math. 62 (2009), no. 3, 305–321.
 
\bibitem{Y}
Y. Yuan, 
\textit{A Bernstein problem for special Lagrangian equations}, 
Invent. Math. 150 (2002), no. 1, 117–125.


\end{thebibliography}
\end{document}